\documentclass[11pt]{amsart}
\usepackage{amsfonts,amssymb,amsmath,amsthm}
\usepackage{url}
\usepackage{enumerate}
\usepackage[dvipdfmx]{graphicx}

\newtheorem{theorem}{Theorem}[section]
\newtheorem{example}[theorem]{Example}
\newtheorem{lemma}[theorem]{Lemma}

\theoremstyle{definition}
\newtheorem{definition}[theorem]{Definition}

\numberwithin{equation}{section}

\author{Yujie Gu}
\address{
Department of Electrical Engineering--Systems, Tel Aviv University, Tel Aviv 6997801, Israel}
\email{guyujie2016@gmail.com}

\author{Shohei Satake}
\address{
Graduate School of System Informatics, Kobe University \\
Rokkodai 1-1, Nada, Kobe, 657-8501, JAPAN}
\email{155x601x@stu.kobe-u.ac.jp}



\begin{document}

\title[On $2$-parent-identifying set systems of block size $4$]
{On $2$-parent-identifying set systems \\of block size $4$}

\maketitle
\begin{abstract}
Parent-identifying set system is a kind of combinatorial structures with applications to broadcast encryption.
In this paper we investigate the maximum number of blocks $I_2(n,4)$ in a $2$-parent-identifying set system with ground set size $n$ and block size $4$.
The previous best-known lower bound states that $I_2(n,4)=\Omega(n^{4/3+o(1)})$. 
We improve this lower bound by showing that $I_2(n,4)= \Omega(n^{3/2-o(1)})$ using techniques in additive number theory.
\end{abstract}

\section{Introduction}
\label{intro}

Traitor tracing was introduced for broadcast encryption in order to protect the copyrighted digital contents~\cite{Chor1994,Chor2000}.
Over the recent decades, several kinds of combinatorial structures which could be applied for the key-distribution schemes against the piracy were proposed and extensively investigated, see \cite{Boneh1998,Chor2000,Collins2009,Hollmann1998,Stinson1998} for example.
In this paper our discussion is based on the combinatorial model introduced in~\cite{Stinson1998}, which could be briefly described in the following.

A dealer, who possesses the copyright of the data, has a set $\mathcal{X}$ of $n$ base decryption keys.
The dealer would assign each authorized user, who purchased the copyright of data, $k$ based keys (i.e. a $k$-subset of $\mathcal{X}$), which, based on a threshold secret sharing scheme, could be used to decrypt the encrypted contents~\cite{Stinson1998}.
In this setting, we could assume an $(n,k)$ \textit{set system} $(\mathcal{X},\mathcal{B})$ where $\mathcal{X}$ is the ground set of $n$ base keys, and $\mathcal{B}$ is a family of $k$-subsets of $\mathcal{X}$ representing all the authorized users.
In a set system $(\mathcal{X},\mathcal{B})$, each element of $\mathcal{X}$ is called a \textit{point}, and each element of $\mathcal{B}$ is referred to as a \textit{block}.
A $t$-collusion means that $t$ dishonest users (traitors) $B_1,\ldots,B_t\in \mathcal{B}$ work together to generate a $k$-subset (pirate)
$T\subseteq \cup_{1\le i\le t}B_i$ and illegally redistribute $T$ to the unauthorized users.
To hinder the illegal redistribution of the decryption key,
once such pirate $T$ is confiscated, the dealer would like to trace back to at least one or more traitors in the coalition.
This requires that the set system $(\mathcal{X},\mathcal{B})$ should have some desired properties.
Parent-identifying set system, which was proposed in~\cite{Collins2009} as a variant of codes with the identifiable parent property (i.e. parent-identifying codes) in~\cite{Hollmann1998}, could provide a kind of traceability as follows.

\begin{definition}\label{defIPPS}
A \textit{$t$-parent-identifying set system}, denoted as $t$-IPPS$(n,k)$, is a pair $(\mathcal{X},\mathcal{B})$ such that $|\mathcal{X}|=n$, $\mathcal{B}\subseteq \binom{\mathcal{X}}{k}=\{F\subseteq \mathcal{X}: |F|=k\}$, with the property that for any $k$-subset $T\subseteq \mathcal{X}$, either $P_t(T)$ is empty, or
\begin{equation*}
\bigcap_{\mathcal{P}\in P_t(T)}\mathcal{P}\ne \emptyset,
\end{equation*}
where
\begin{equation*}
P_t(T)=\left\{\mathcal{P}\subseteq \mathcal{B}:\ |\mathcal{P}|\le t,\ T\subseteq \bigcup_{B\in \mathcal{P}}B\right\}.
\end{equation*}
\end{definition}

For a set $T$ and a subset $\mathcal{P}\subseteq \mathcal{B}$, if $T\subseteq \bigcup_{B\in \mathcal{P}}B$, then we say $\mathcal{P}$ is a \textit{possible parent set} of $T$.
As we can see, a key-distribution scheme based on a $t$-parent-identifying set system could identify at least one traitor in a collusion with at most $t$ colluders.
Indeed, if a pirate $T$ is captured, one could first find out $P_t(T)$ which is the collection of all the possible parent sets of $T$ with cardinality at most $t$. Then the guys who exist in every $\mathcal{P}\in P_t(T)$ must be colluders for generating pirate $T$.

The number of blocks $B \in \mathcal{B}$ is called the \textit{size} of this $t$-IPPS$(n,k)$.
Denote $I_t(n,k)$ as the maximum size of a $t$-IPPS$(n,k)$.
An $(n,k)$ set system $(\mathcal{X},\mathcal{B})$ which is a $t$-IPPS$(n,k)$ is called \textit{optimal} if it has size $I_t(n,k)$. Notice that in the practical application, $I_t(n,k)$ corresponds to the maximum number of users which could be accommodated in the collusion-resistant system. In what follows, we are interested with the value of $I_t(n,k)$.

Throughout the paper we use the standard asymptotic notations. Let $f(n)>0$, $g(n)>0$ for any positive integer $n$. Then (1) $f(n)=o(g(n))$ as $n \to \infty$ if $\displaystyle{\lim_{n\to \infty }}  {f(n)}/{g(n)}=0$;
(2) $f(n)=O(g(n))$ as $n \to \infty$ if $\displaystyle{\limsup_{n\to \infty }} {f(n)}/{g(n)}< \infty$;
(3) $f(n)=\Omega(g(n))$ as  $n \to \infty$ if $\displaystyle{\liminf_{n\to \infty }}  {f(n)}/{g(n)}> 0$.
We will omit the suffix ``as $n \to \infty$" when it is clear from the context.

In the literature, the best known general upper bound for $I_t(n,k)$ is due to~\cite{GM2016}.

\begin{theorem}[\cite{GM2016}]\label{newupboundIPPS}
Let $n\ge k\ge 2,\ t\ge 2$ be integers. Then
\begin{equation*}
I_t(n,k)\le \binom{n}{\lceil\frac{k}{\lfloor t^2/4\rfloor+t}\rceil}=O(n^{\lceil\frac{k}{\lfloor t^2/4\rfloor+t}\rceil}),
\end{equation*}
as $n\to \infty$.
\end{theorem}

The best known general lower bound for $I_t(n,k)$ is from~\cite{GCKM} via a probabilistic method.

\begin{theorem}[\cite{GCKM}]\label{newlowIPPS}
Let $k$ and $t$ be positive integers such that $t\ge 2$. Then there exists a constant $c$, depending only on $k$ and $t$, with the following property. For any sufficiently large integer $n$, there exists a $t$-IPPS$(n,k)$ with size at least $c n^{\frac{k}{\mu-1}}$,
that is, $I_t(n,k)\ge c n^{\frac{k}{\mu-1}}$, where $\mu=\lfloor(\frac{t}{2}+1)^2\rfloor$.
\end{theorem}

For a $2$-IPPS$(n,k)$, the following lemma provides an equivalent description as Definition~\ref{defIPPS}.

\begin{lemma}[\cite{GCKM}]\label{2IPPSabc}
An $(n,k)$ set system $(\mathcal{X},\mathcal{B})$ is a $2$-IPPS$(n,k)$ if and only if the following cases hold.
\begin{description}
  \item[(IPPSa)]\quad For any three distinct blocks $A,B,C\in \mathcal{B}$, we have
  \begin{equation*}
  |(A\cup B)\cap (A\cup C)\cap (B\cup C)|< k.
  \end{equation*}
  \item[(IPPSb)]\quad For any four distinct blocks $A_1,A_2,B_1,B_2\in \mathcal{B}$, we have
  \begin{equation*}
  |(A_1\cup A_2)\cap (B_1\cup B_2)|< k.
  \end{equation*}
\end{description}
\end{lemma}

In this paper we shall investigate the value $I_2(n,4)$ for $2$-IPPS$(n,4)$, especially, when $n$ is sufficiently large.
Notice that it is reasonable to consider large $n$ and relatively small $k$. In fact, the dealer needs a large set of base keys to accommodate amounts of authorized users, however, each authorized user is usually assigned with a limited number of base keys which are used as the user's inputs to the decryption devices~\cite{GCKM}.

In the literature, an upper bound of $I_2(n,4)$ which is better than that of Theorem~\ref{newupboundIPPS} was proven in~\cite{GCKM} using a graph theoretic approach.

\begin{lemma}[\cite{GCKM}]\label{lemma-upper-2IPPS}
$I_2(n,4)=o(n^2)$.
\end{lemma}
Also a lower bound of $I_2(n,4)$ which is slightly better than that of Theorem~\ref{newlowIPPS} can be found in~\cite{STamo}.
\begin{lemma}[\cite{STamo}]\label{coro-lower}
$I_2(n,4)=\Omega(n^{4/3+o(1)})$.
\end{lemma}

The objective of this paper is to improve this lower bound of $I_2(n,4)$ for $2$-IPPS$(n,4)$ using techniques in additive number theory.
Specifically, we will show that $I_2(n,4)=\Omega(n^{3/2-o(1)})$ by providing a construction for $2$-IPPS$(n,4)$.

The paper is organized as follows. In Section \ref{sec-addi}, we first present the useful lemmas in additive number theory. Our new construction for $2$-IPPS$(n,4)$ is exhibited in Section~\ref{sec-construction}. 
A discussion on IPPS and the known parent-identifying codes is provided in Section~\ref{sec-discussion}.  
Finally concluding remarks are made in Section~\ref{sec-conclude}.

\section{Additive number theory}
\label{sec-addi}

A linear equation with integer coefficients
\begin{equation}\label{eq-integer}
\sum_{1\le i\le r} a_ix_i=0
\end{equation}
in the $r$ unknowns $x_i$ is \textit{homogeneous} if $\sum_{1\le i\le r} a_i=0$.
It is readily seen that the homogeneous equation (\ref{eq-integer}) has the translation invariance property, that is, if $(x_1,  \ldots,x_r)$ is a solution of (\ref{eq-integer}), then for any $u\in \mathbb{Z}$, $(x_1+u,\ldots,x_r+u)$ is also a solution of (\ref{eq-integer}).
Considering a set $S\subseteq [n]=\{1,\ldots,n\}$, we say $S$ has {\it no non-trivial solution} to (\ref{eq-integer}) if
whenever $s_i\in S$ and $\sum_{1\le i\le r} a_is_i=0$, it follows that all $s_i$ are equal.
Notice that, by the translation invariance, if $S$ has no non-trivial solution to (\ref{eq-integer}),
then the same holds for any shift $(S+u)\cap [n]$, where $u\in \mathbb{Z}$ and $S+u=\{s+u: s\in S\}$.

The following lemma was proved in~\cite[Corollary 3.3]{Alon2001}. Throughout the paper the logarithm is taken in base $2$.
\begin{lemma}[\cite{Alon2001}]\label{lemma-alon}
For $q=\lceil 2^{\sqrt{\log m}}\rceil$ there exists 
a set $S_0\subseteq [m]$, $|S_0|\ge \frac{m}{2^{O(\log ^{3/4} m) }}$  with no non-trivial solution to any of the following equations
\begin{align}
    &2x+3y+qz-(q+5)w=0,\label{eq-1}\\
    &5x+(q+3)y-3z-(q+5)w=0,\label{eq-2}\\
    &5x+qy-2z-(q+3)w=0.\label{eq-3}
\end{align}
\end{lemma}

In addition, we need the following result from~\cite[Theorem 7.3]{Ruzsa}.
\begin{lemma}[\cite{Ruzsa}]\label{lemma-rusza}
There exists a set $S_1\subseteq [m]$, $|S_1|\ge \frac{\sqrt{m}}{2^{O(\log ^{1/2} m)}}$ with no non-trivial solution to
\begin{equation}\label{eq-4}
ax+by=az+bw,
\end{equation}
where $a,b$ are positive integers.
\end{lemma}

Combining the above Lemmas \ref{lemma-alon} and \ref{lemma-rusza}, we have
\begin{lemma}\label{lemma-main}
There exists a set $S\subseteq [m]$ such that
\begin{equation}\label{eq-size-S}
|S|\ge \frac{\sqrt{m}}{2^{O(\log ^{3/4} m) }}
\end{equation}
with no non-trivial solution to any of the equations (\ref{eq-1}), (\ref{eq-2}), (\ref{eq-3}), (\ref{eq-4}).
\end{lemma}
\begin{proof}
We shall prove this lemma using a probabilistic method. 
Let $S_0$ and $S_1$ be the sets given in Lemma~\ref{lemma-alon} and \ref{lemma-rusza} respectively.
Take an integer $-m\le u\le m$ randomly and uniformly. 
By the translation invariance,
\begin{equation}\label{eq-intersection}
S=(S_0+u)\cap S_1
\end{equation}
has no non-trivial solution to any of the equations (\ref{eq-1}), (\ref{eq-2}), (\ref{eq-3}) and (\ref{eq-4}).
Now we argue the cardinality of $S$. Notice that each $s\in S_1$ has probability at least $2^{-O(\log ^{3/4}m)}$ to lie in the intersection (\ref{eq-intersection}).
Then by the linearity of expectation,
there exists a set $S$ such that
\begin{equation}
|S|\ge \frac{\sqrt{m}}{2^{O(\log ^{1/2} m)}} \cdot 2^{-O(\log ^{3/4}m)}=\frac{\sqrt{m}}{2^{O(\log ^{3/4} m) }}.
\end{equation}
This completes the proof.
\end{proof}

\section{A construction of $2$-IPPS$(n,4)$}
\label{sec-construction}

In this section we shall provide a construction for $2$-IPPS$(n,4)$ using Lemma~\ref{lemma-main}. Our construction is a modification of the one in~\cite{Alon2001} for codes with identifiable parent property. Also a discussion on the difference between IPPS considered in this paper and codes with identifiable parent property studied such as in~\cite{Alon2001} is referred to Section~\ref{sec-discussion}. 

\begin{theorem}\label{theorem-mian}
$I_2(n,4)=\Omega(n^{3/2-o(1)})$.
\end{theorem}

\begin{proof}
To prove this theorem, we shall show that for any $\epsilon >0$ and sufficiently large $n$, there exists a $2$-IPPS$(n,4)$ with size $n^{3/2-\epsilon}$.

Set $q=\lceil 2^{\sqrt{\log m}}\rceil$ and
\begin{equation}\label{eq-n}
n=4(q+6)m.
\end{equation}
Let $S\subseteq [m]$ be a set shown in Lemma~\ref{lemma-main}, which satisfies (\ref{eq-size-S}) and has no non-trivial solution to any of the equations (\ref{eq-1}), (\ref{eq-2}), (\ref{eq-3}), (\ref{eq-4}).
Define a set system $(\mathcal{X},\mathcal{B})$ with $\mathcal{X}=[4]\times [(q+6)m]$ and
\begin{equation}\label{eq-define-B}
\mathcal{B}=\Big\{\{(1,p),(2,p+2s),(3,p+5s),(4,p+(q+5)s)\}:\, 1\le p\le m, s\in S\Big\}.
\end{equation}
It is easy to see that
\begin{equation}
|\mathcal{B}|\ge m\frac{\sqrt{m}}{2^{O(\log ^{3/4} m)}}= \frac{m^{3/2}}{2^{O(\log ^{3/4} m)}}.
\end{equation}
Consequently, for any $\epsilon>0$ there exists an $m_0(\epsilon)>0$ (and hence an $n_0(\epsilon)>0$ by (\ref{eq-n})) such that for any $n>n_0(\epsilon)$, we have
\begin{equation}
|\mathcal{B}|\ge n^{3/2-\epsilon}.
\end{equation}

Now it suffices to claim that the set system $(\mathcal{X},\mathcal{B})$ defined in (\ref{eq-define-B}) is a $2$-IPPS$(n,4)$.
First notice that for any distinct $B_1,B_2\in \mathcal{B}$ we have \begin{equation}\label{eq-assump-joint}
|B_1\cap B_2|\le 1,
\end{equation}
since the four linear functions (in $p$ and $s$) used in~\eqref{eq-define-B} are pairwise non-collinear.
%
Hence for any three distinct blocks $A,B,C\in \mathcal{B}$ we have
\begin{equation}
\begin{split}
&|(A\cup B)\cap (A\cup C)\cap (B\cup C)|\\
&=|(A\cap B) \cup (A\cap C) \cup(B\cap C)|\\
&\le |A\cap B| + |A\cap C| + |B\cap C|\\
&\le 3 <4.
\end{split}
\end{equation}
This implies that $(\mathcal{X},\mathcal{B})$ satisfies (IPPSa) in Lemma~\ref{2IPPSabc}.
It remains to show that the case (IPPSb) also holds.

Suppose $A_1,A_2,B_1,B_2$ are four distinct blocks in $\mathcal{B}$ with
\begin{equation}
\begin{split}
A_1&=\{(1,p_1), (2, p_1+2x), (3,p_1+5x), (4,p_1+(q+5)x)\},\\
A_2&=\{(1,p_2), (2, p_2+2y), (3,p_2+5y), (4,p_2+(q+5)y)\},\\
B_1&=\{(1,p_3), (2, p_3+2z), (3,p_3+5z), (4,p_3+(q+5)z)\},\\
B_2&=\{(1,p_4), (2, p_4+2w), (3,p_4+5w), (4,p_4+(q+5)w)\},
\end{split}
\end{equation}
where $1\le p_1,p_2,p_3,p_4\le m$ and $x,y,z,w\in S$. Now we shall show that
\begin{equation}
|(A_1\cup A_2)\cap (B_1\cup B_2)|<4.
\end{equation}
If not, we might assume there exists a $4$-subset $T\subseteq [4]\times [(q+6)m]$ with
\begin{equation}\label{eq-T}
T=\{(i,\alpha),(j,\beta),(k,\gamma),(l,\delta)\} \subseteq (A_1\cup A_2)\cap (B_1\cup B_2),
\end{equation}
where $1\le i,j,k,l\le 4$ might be the same.
Now we would like to derive contradictions.

First notice that 
\begin{align*}
    4=|T|\overset{~\eqref{eq-T}}{\le} &| (A_1\cup A_2)\cap (B_1\cup B_2)|\\
    =\ &\Big|\bigcup_{u,v\in \{1,2\}} (A_u\cap B_v)\Big|\\
    \le \ & \sum_{u,v\in \{1,2\}} |A_u\cap B_v|
    \overset{~\eqref{eq-assump-joint}}{\le } 4.
\end{align*}
This implies that every set $ A_u\cap B_v$, $u,v\in \{1,2\}$ consists of a single point and these points are distinct. In other words, we have 
\begin{equation}
|T\cap A_1|=|T\cap A_2|=|T\cap B_1|=|T\cap B_2|=2.
\end{equation}

Recall that $T\subseteq A_1\cup A_2$, 
without loss of generality, we may assume
\begin{equation}
\begin{split}
&T\cap A_1=\{(i,\alpha),(j,\beta)\},\\
&T\cap A_2=\{(k,\gamma),(l,\delta)\},
\end{split}
\end{equation}
where $1\le i,j,k,l\le 4$, $i\ne j$ and $k\ne l$.
The following analysis is divided into cases according to the intersection of $\{i,j\}$ and $\{k,l\}$.

\textit{Case 1.} Consider $\{i,j\}\cap \{k,l\}=\emptyset$. By the symmetry of $A_1$ and $A_2$, we only need to consider the following three cases
\begin{eqnarray}\label{eq-cases}
    \begin{array}{l}
      \text{(case 1.1)} \ \ \ i=1,\ j=2,\ k=3,\ l=4, \\
      \text{(case 1.2)} \ \ \ i=1,\ j=3,\ k=2,\ l=4, \\
      \text{(case 1.3)} \ \ \ i=1,\ j=4,\ k=2,\ l=3.
    \end{array}
\end{eqnarray}

\textit{Case 1.1}
If $i=1$, $j=2$, $k=3$, $l=4$, then it follows that, up to symmetry of $B_1$ and $B_2$, either
\begin{eqnarray}\label{eq-case1.1-1}
  \left\{
    \begin{array}{l}
      \alpha=p_1=p_3 \\
      \beta=p_1+2x=p_4+2w \\
      \gamma=p_2+5y=p_3+5z \\
      \delta=p_2+(q+5)y=  p_4+(q+5)w
    \end{array}
  \right.
\end{eqnarray}
or
\begin{eqnarray}\label{eq-case1.1-2}
  \left\{
    \begin{array}{l}
      \alpha=p_1=p_3 \\
      \beta=p_1+2x=p_4+2w \\
      \gamma=p_2+5y=p_4+5w \\
      \delta=p_2+(q+5)y=  p_3+(q+5)z.
    \end{array}
  \right.
\end{eqnarray}
From (\ref{eq-case1.1-1}), we obtain
\begin{equation}\label{eq-case1.1-eq}
2x-qy-5z+(q+3)w=0.
\end{equation}
Since $x,y,z,w\in S$ and $S$ has no non-trivial solution to equation (\ref{eq-3}) (and also (\ref{eq-case1.1-eq})), we have $x=y=z=w$. Together with $p_1=p_3$ in (\ref{eq-case1.1-1}), we get $A_1=B_1$, a contradiction to the assumption that $A_1,A_2,B_1,B_2$ are distinct.
Similarly, from (\ref{eq-case1.1-2}), we have
\begin{equation}\label{eq-case1.1.2-eq}
2x+qy-(q+5)z+3w=0.
\end{equation}
Since $x,y,z,w\in S$ and $S$ has no non-trivial solution to equation (\ref{eq-1}) (and also (\ref{eq-case1.1.2-eq})), we obtain $x=y=z=w$, which together with $p_1=p_3$ in (\ref{eq-case1.1-2}) results $A_1=B_1$, a contradiction to our assumption that $A_1,A_2,B_1,B_2$ are distinct.

\textit{Case 1.2} If $i=1$, $j=3$, $k=2$, $l=4$, then by the symmetry of $B_1$ and $B_2$, we have
\begin{eqnarray}\label{eq-case1.2-1}
  \left\{
    \begin{array}{l}
      \alpha=p_1=p_3 \\
      \beta=p_1+5x=p_4+5w \\
      \gamma=p_2+2y=p_3+2z \\
      \delta=p_2+(q+5)y=  p_4+(q+5)w
    \end{array}
  \right.
\end{eqnarray}
or
\begin{eqnarray}\label{eq-case1.2-2}
  \left\{
    \begin{array}{l}
      \alpha=p_1=p_3 \\
      \beta=p_1+5x=p_4+5w \\
      \gamma=p_2+2y=p_4+2w \\
      \delta=p_2+(q+5)y=  p_3+(q+5)z.
    \end{array}
  \right.
\end{eqnarray}
For equations (\ref{eq-case1.2-1}), one could derive a contradiction via equation (\ref{eq-3}) in a similar way as for equations (\ref{eq-case1.1-1}).
According to (\ref{eq-case1.2-2}), we obtain
\begin{equation}\label{eq-case1.2-eq}
5x+(q+3)y-(q+5)z-3w=0.
\end{equation}
Since $x,y,z,w\in S$ and $S$ has no non-trivial solution to equation (\ref{eq-2}) (and also (\ref{eq-case1.2-eq})),
we have $x=y=z=w$, which together with $p_1=p_3$ in (\ref{eq-case1.2-2}) shows that $A_1=B_1$, a contradiction to our assumption that $A_1$ and $B_1$ are distinct.

\textit{Case 1.3} If $i=1$, $j=4$, $k=2$, $l=3$, then by the symmetry of $B_1$ and $B_2$, we have
\begin{eqnarray}\label{eq-case1.3-1}
  \left\{
    \begin{array}{l}
      \alpha=p_1=p_3 \\
      \beta=p_1+(q+5)x=p_4+(q+5)w \\
      \gamma=p_2+2y=p_3+2z \\
      \delta=p_2+5y=  p_4+5w
    \end{array}
  \right.
\end{eqnarray}
or
\begin{eqnarray}\label{eq-case1.3-2}
  \left\{
    \begin{array}{l}
      \alpha=p_1=p_3 \\
      \beta=p_1+(q+5)x=p_4+(q+5)w \\
      \gamma=p_2+2y=p_4+2w \\
      \delta=p_2+5y=  p_3+5z.
    \end{array}
  \right.
\end{eqnarray}
For equations (\ref{eq-case1.3-1}), one could derive a contradiction via equation (\ref{eq-1}) as the way for equations (\ref{eq-case1.1-2}).
Also for equations (\ref{eq-case1.3-2}), a contradiction can be derived via equation (\ref{eq-2}) following the way for equations (\ref{eq-case1.2-2}).

\textit{Case 2.} Consider $|\{i,j\}\cap \{k,l\}|=1$. Recall from~\eqref{eq-assump-joint} that any two distinct blocks in $\mathcal{B}$ have at most one common point.
If $i=k=1$, $j=2$, $l=3$, by the symmetry of $B_1$ and $B_2$, we may assume
 \begin{eqnarray}\label{eq-case2}
  \left\{
    \begin{array}{l}
      \alpha=p_1=p_3 \\
      \beta=p_1+2x=p_4+2w \\
      \gamma=p_2=p_4 \\
      \delta=p_2+5y=  p_3+5z,
    \end{array}
  \right.
\end{eqnarray}
yielding
\begin{equation}\label{eq-case2-eq}
2x+5y=2w+5z.
\end{equation}
Since $x,y,z,w\in S$ and $S$ has no non-trivial solution to equation (\ref{eq-4}) (and hence equation (\ref{eq-case2-eq})), we have
$x=y=z=w$. Together with $p_1=p_3$ and $p_2=p_4$ in (\ref{eq-case2}), we get $A_1=B_1$ and $A_2=B_2$, a contradiction to our assumption that $A_1,A_2,B_1,B_2$ are pairwise distinct.
Accordingly, for any $1\le i,j,k,l\le 4$ such that $|\{i,j\}\cap \{k,l\}|=1$, one could derive a contradiction via equation (\ref{eq-4}) in the same way.

\textit{Case 3.} Consider $|\{i,j\}\cap \{k,l\}|=2$. Without loss of generality, assume $i=k=3$, $j=l=4$. By the symmetry, we could have
 \begin{eqnarray}\label{eq-case3}
  \left\{
    \begin{array}{l}
      \alpha=p_1+5x=p_3+5z \\
      \beta=p_1+(q+5)x=p_4+(q+5)w \\
      \gamma=p_2+5y=p_4+5w \\
      \delta=p_2+(q+5)y=  p_3+(q+5)z,
    \end{array}
  \right.
\end{eqnarray}
resulting
\begin{equation}\label{eq-case3-eq}
x+y=z+w.
\end{equation}
Recall that $x,y,z,w\in S$ and $S$ has no non-trivial solution to equation (\ref{eq-4}) (and also equation (\ref{eq-case3-eq})). Hence we have $x=y=z=w$, which, together with (\ref{eq-case3}), implies $A_1=B_1$ and $A_2=B_2$, a contradiction to our assumption that $A_1,A_2,B_1,B_2$ are pairwise distinct. In the meanwhile, for any other values of $1\le i,j,k,l\le 4$ such that $|\{i,j\}\cap \{k,l\}|=2$, one could argue in a similar way and derive a contradiction via equation (\ref{eq-4}).

Based on the foregoing, the cases (IPPSa) and (IPPSb) in Lemma~\ref{2IPPSabc} hold for the set system $(\mathcal{X},\mathcal{B})$ defined in (\ref{eq-define-B}), implying that $(\mathcal{X},\mathcal{B})$ is a $2$-IPPS$(n,4)$.
This completes the proof.
\end{proof}

\section{Discussion on IPP set systems and IPP codes}
\label{sec-discussion}

In this section, we shall discuss the similarities and differences between parent-identifying set systems and parent-identifying codes. 
We first recall the notion of parent-identifying codes which was proposed in~\cite{Hollmann1998}.

\begin{definition}\label{defIPPC}
A $q$-ary \textit{$t$-parent-identifying code}, denoted as $t$-IPPC$(n,q)$, is a set $\mathcal{C}\subseteq [q]^n$ with the property that for any
word $\mathbf{d}\in [q]^n$, either $S_t(\mathbf{d})$ is empty, or
\begin{equation*}
\bigcap_{\mathcal{S}\in S_t(\mathbf{d})}\mathcal{S}\ne \emptyset,
\end{equation*}
where
\begin{equation*}
S_t(\mathbf{d})=\left\{\mathcal{S}\subseteq \mathcal{C}:\ \mathbf{d}\in \text{desc}(\mathcal{S}), |\mathcal{S}|\le t \right\}
\end{equation*}
and 
\begin{equation*}
    \text{desc}(\mathcal{S})=\{\mathbf{x}=(x_1,\ldots,x_n)\in [q]^n: x_i\in \{s_i: \mathbf{s}\in \mathcal{S}\}, \forall\, i \}.
\end{equation*}
\end{definition}

On the one hand, IPP codes and IPP set systems share several similar features since they are both defined by using the parent-identifying property (but in different scenarios). A unified perspective for analyzing these structures is referred to~\cite{GCKM}. 
The following is one of their typical common performances.

\begin{lemma}[\cite{GCKM}] Let $\mu=\lfloor (t/2+1)^2\rfloor$. 
\begin{itemize}
    \item[(1)] 
    A code $\mathcal{C}\subseteq [q]^n$ is a $t$-IPPC$(n,q)$ if and only if every subcode $\mathcal{C}'\subseteq \mathcal{C}$ such that $|\mathcal{C}'|\le \mu$ is a $t$-IPPC$(n,q)$.
    \item[(2)]  A set system $(\mathcal{X},\mathcal{B})$ is a $t$-IPPS$(n,k)$ if and only if every $(\mathcal{X},\mathcal{B}')$, where $\mathcal{B}'\subseteq \mathcal{B}$ such that $|\mathcal{B}'|\le \mu$, is a $t$-IPPS$(n,k)$.
\end{itemize}
\end{lemma}

On the other hand, as pointed out in~\cite{Stinson1998}, IPP set systems could be regarded as a kind of generalization of IPP codes in the sense that, according to the model in Section~\ref{intro}, IPP codes are considered on the basis of an $n$-out-of-$n$ threshold secret sharing scheme and IPP set systems are concerned with respect to a more general $k$-out-of-$n$ threshold secret sharing scheme.  
Then one nature question is ``Can we derive IPP set systems directly from IPP codes?". At present we do not have a positive or negative answer in general. One feasible way of constructing set systems from codes, as we did in~\eqref{eq-define-B}, is via the so-called Kautz-Singleton construction, which was invented more than fifty years ago, for superimposed codes~\cite{KS1964}, now better known as cover-free codes or families~\cite{EFF1982}. 
However only $t$-IPP codes plugging into the Kautz-Singleton construction cannot directly yield $t$-IPP set systems since there are more requirements in IPP set systems than in IPP codes (see Example~\ref{example-1} below). 

\begin{example}\label{example-1} 
The ternary Hamming code of length $4$ 
\begin{equation*}
    \mathcal{C}=\{1111,1222,1333,2123,2231,2312,3132,3213,3321\}\subseteq \{1,2,3\}^4
\end{equation*}
is claimed as a $2$-IPPC$(4,3)$ in~\cite{Hollmann1998}. Plugging into the Kautz-Singleton construction, we obtain a set system $(\mathcal{X},\mathcal{B})$ where $\mathcal{X}=[4]\times [3]$ and 
\begin{align*}
    \mathcal{B}=\big\{&B_1=\big( (1,1), (2,1), (3,1), (4,1)   \big),\  B_2=\big( (1,1), (2,2), (3,2), (4,2)   \big),\\
    &B_3=\big( (1,1), (2,3), (3,3), (4,3)   \big),\  B_4=\big( (1,2), (2,1), (3,2), (4,3)   \big),\\
    &B_5=\big( (1,2), (2,2), (3,3), (4,1)   \big),\  B_6=\big( (1,2), (2,3), (3,1), (4,2)   \big),\\
    &B_7=\big( (1,3), (2,1), (3,3), (4,2)   \big),\  B_8=\big( (1,3), (2,2), (3,1), (4,3)   \big),\\
    &B_9=\big( (1,3), (2,3), (3,2), (4,1)   \big)
    \big\}.
\end{align*}
Now we claim that $(\mathcal{X},\mathcal{B})$ is not a $2$-IPPS$(12,4)$. In fact, considering four distinct blocks $B_1,B_2,B_4,B_5\in \mathcal{B}$, we have 
\begin{equation*}
    \{(1,1),(1,2),(3,2),(4,1)\}\subseteq (B_1\cup B_4) \cap (B_2\cup B_5),
\end{equation*}
a contradiction to the requirement (IPPSb) in Lemma~\ref{2IPPSabc}. \qed 
\end{example}

As can be seen, if one would like to derive IPP set systems from IPP codes, more techniques other than the Kautz-Singleton construction are required.
Indeed, our construction for $2$-IPPS$(n,4)$ in Theorem~\ref{theorem-mian} is with the help of Lemmas~\ref{lemma-rusza} and~\ref{lemma-main} in addition to the IPP code. It would be also interesting to find a general way of converting IPP codes to IPP set systems. 
Some discussions about this as well as the relation between IPP set systems and binary constant weight codes can be found in~\cite{EK2017}. Also some variants of IPP codes and IPP set systems are referred to~\cite{E2019} and~\cite{Kaba2019}.


\section{Concluding remarks}
\label{sec-conclude}

In this paper we provided a construction for $2$-IPPS$(n,4)$ using techniques in additive number theory.
This gives a lower bound on the maximum size of a $2$-IPPS$(n,4)$, that is $I_2(n,4)=\Omega(n^{3/2-o(1)})$, which improves the best existing result $I_2(n,4)=\Omega(n^{4/3+o(1)})$.
Together with the best known upper bound in Lemma~\ref{lemma-upper-2IPPS}, we have
$$\Omega(n^{3/2-o(1)})=I_2(n,4)=o(n^2).$$
It would be of interest to continue to narrow the gap between the upper and lower bounds, especially, on the order of magnitude of $I_2(n,4)$, as well as to explore a generalization of the construction for $2$-IPPS$(n,4)$ in Theorem~\ref{theorem-mian}.

It is worth noting that in \cite[Theorem 3.2]{Ruzsa}, Ruzsa proved that for every $S_1 \subseteq [m]$ with no non-trivial solution to equation (\ref{eq-4}), we have $$|S_1| = O(\sqrt{m}).$$
This implies an upper bound on the cardinality of the set $S \subseteq [m]$ with no non-trivial solution to any of the equations (\ref{eq-1}), (\ref{eq-2}), (\ref{eq-3}) and (\ref{eq-4}), that is, $$|S|\le |S_1| = O(\sqrt{m}).$$
Hence we could see that the set $S \subseteq[m]$ shown in Lemma~\ref{lemma-main} has the (almost) best possible order of magnitude.
Based on this,
using the same construction as in our proof of Theorem~\ref{theorem-mian} cannot give better estimates such as $I_2(n, 4) = \Omega(n^{3/2+o(1)})$.

\section*{acknowledgements}
The authors express their gratitude to the two anonymous
reviewers for their detailed and constructive comments which
are very helpful to the improvement of the presentation of this
paper, and to the associate editor for
his/her excellent editorial job. 
S. Satake has been supported by Grant-in-Aid for JSPS Fellows 18J11282 of the Japan Society for the Promotion of Science.

\bibliographystyle{spphys}       


\end{document}